\documentclass[a4paper,12pt]{article}

\usepackage{amsfonts}
\usepackage{amscd,color}
\usepackage{amsmath,amsfonts,amssymb,amscd}
\usepackage{indentfirst,graphicx,epsfig}
\usepackage{graphicx,psfrag}
\input{epsf}
\usepackage{graphicx}
\usepackage{epstopdf}
\usepackage{caption}

\newtheorem{thm}{Theorem}[section]
\newtheorem{df}[thm]{Definition}
\newtheorem{problem}[thm]{Problem}

\newtheorem{lem}[thm]{Lemma}

\newtheorem{cor}[thm]{Corollary}

\newenvironment {proof} {\noindent{\em Proof.}}{\hspace*{\fill}$\Box$\par\vspace{4mm}}

\title{\bf Degree sum conditions for graphs\\ to have proper connection number 2\footnote{Supported by NSFC No.11371205, 11531011.}}

\author{{\small Hong Chang, Zhong Huang, Xueliang Li} \\
{\small  Center for Combinatorics and LPMC}\\
{\small Nankai University, Tianjin 300071, P.R. China}\\
{\small Email: changh@mail.nankai.edu.cn, 2120150001@mail.nankai.edu.cn, lxl@nankai.edu.cn}\\
}
\date{}
\begin{document}
\maketitle
\begin{abstract}
A path $P$ in an edge-colored graph $G$ is a \emph{proper path} if no two adjacent edges of $P$ are colored with the same color. The graph $G$ is \emph{proper connected} if, between every pair of vertices, there exists a proper path in $G$. The \emph{proper connection number} $pc(G)$ of a connected graph $G$ is defined as the minimum number of colors to make $G$ proper connected. In this paper, we study the degree sum condition for a general graph or a bipartite graph to have proper connection number 2. First, we show that if $G$ is a connected noncomplete graph of order $n\geq 5$ such that $d(x)+d(y)\geq \frac{n}{2}$ for every pair of nonadjacent vertices $x,y\in V(G)$, then $pc(G)=2$ except for three small graphs on 6, 7 and 8 vertices. In addition, we obtain that if $G$ is a connected bipartite graph of order $n\geq 4$ such that $d(x)+d(y)\geq \frac{n+6}{4}$ for every pair of nonadjacent vertices $x,y\in V(G)$, then $pc(G)=2$. Examples are given to show that the above conditions are best possible. \\[2mm]
\textbf{Keywords:} proper connection number; proper connection coloring; bridge-block tree; degree sum condition.\\
\textbf{AMS subject classification 2010:} 05C15, 05C40, 05C07.\\
\end{abstract}

\section{Introduction}

All graphs in this paper are undirected, finite and simple. We follow \cite{BM} for graph theoretical notation and terminology not described here. Let $G$ be a graph. We use $V(G), E(G), \left|G\right|, \Delta(G)$ and $\delta(G)$ to denote the vertex set, edge set, number of vertices, maximum degree and minimum degree of $G$, respectively. For any two disjoint subsets $X$ and $Y$ of $V(G)$, we denote by $E_G(X,Y)$ the set of edges of $G$ that have one end in $X$ and the other in $Y$, and denote by $\left|E_G(X,Y)\right|$ the number of edges in $E_G(X,Y)$. For $v\in V(G)$, let $N(v)$ denote the set of neighbours, and $d_H(v)$ denote the degree of $v$ in subgraph $H$ of $G$.

Let $G$ be a nontrivial connected graph with an associated \emph{edge-coloring} $c : E(G)\rightarrow \{1, 2, \ldots, t\}$, $t \in \mathbb{N}$, where adjacent edges may have the same color. If adjacent edges of $G$ are assigned different colors by $c$, then $c$ is a \emph{proper (edge-)coloring}. For a graph $G$, the minimum number of colors needed in a proper coloring of $G$ is referred to as the \emph{edge-chromatic number} of $G$ and denoted by $\chi'(G)$. A path of an edge-colored graph $G$ is said to be a \emph{rainbow path} if no two edges on the path have the same color. The graph $G$ is called \emph {rainbow connected} if every pair of distinct vertices of $G$ is connected by a rainbow path in $G$.  An edge-coloring of a connected graph is a \emph{rainbow connection coloring} if it makes the graph rainbow connected. This concept of rainbow connection of graphs was introduced by Chartrand et al.~\cite{CJMZ} in 2008. The \emph{rainbow connection number} $rc(G)$ of a connected graph $G$ is the smallest number of colors that are needed in order to make $G$ rainbow connected. The readers who are interested in this topic can see \cite{LSS,LS} for a survey.

Inspired by rainbow connection coloring and proper coloring in graphs, Andrews et al.~\cite{ALLZ} and Borozan et al.~\cite{BFGMMMT} introduced the concept of proper-path coloring. Let $G$ be a nontrivial connected graph with an edge-coloring. A path in $G$ is called a \emph{proper path} if no two adjacent edges of the path receive the same color. An edge-coloring $c$ of a connected graph $G$ is a \emph{proper connection  coloring} if every pair of distinct vertices of $G$ are connected by a proper path in $G$. And if $k$ colors are used, then $c$ is called a \emph{proper connection $k$-coloring}. An edge-colored graph $G$ is \emph{proper connected} if any two vertices of $G$ are connected by a proper path. For a connected graph $G$, the minimum number of colors that are needed in order to make $G$ proper connected is called the \emph{proper connection number} of $G$, denoted by \emph{$pc(G)$}. Let $G$ be a nontrivial connected graph of order $n$ and size $m$, then we have that $1\leq pc(G) \leq \min\{\chi'(G), rc(G)\}\leq m$. Furthermore, $pc(G)=1$ if and only if $G=K_n$ and $pc(G)=m$ if and only if $G=K_{1,m}$ as a star of size $m$. For more details, we refer to \cite{GLQ,LLZ,LWY} and a dynamic survey \cite{LC}.

In \cite{ALLZ}, the authors considered many conditions on $G$ which force $pc(G)$ to be small, in particular $pc(G)=2$. Recently, Huang et al. presented minimum degree condition for a graph to have proper connection number 2 in \cite{HLQ}. They showed that if $G$ is a connected noncomplete graph of order $n\geq5$ with $\delta(G)\geq n/4$, then $pc(G)=2$ except for two small graphs on 7 and 8 vertices. In addition, they obtained that if $G$ is a connected bipartite graph of order $n\geq4$ with $\delta(G)\geq \frac{n+6}{8}$, then $pc(G)=2$. It is worth mentioning that the two bounds on the minimum degree in the above results are best possible. On the other hand, in \cite{W}, the authors showed that if $G$ is a graph with $n$ vertices such that  $\delta(G)\geq\frac{n-1}{2}$, then $G$ has a Hamiltonian path. It is also known that if a noncomplete graph $G$ has a Hamiltonian path, then $pc(G)=2$ in \cite{ALLZ}. Hence, we can say that if the graph $G$ is not a complete graph with $\delta(G)\geq\frac{n-1}{2}$, then $pc(G)=2$. These results naturally lead the following two problems.

\begin{problem}\label{problem1}
Is there a constant $\frac{1}{2}\leq \alpha < 1$, such that if $d(x)+d(y)\geq \alpha n$ for every pair of nonadjacent vertices $x,y$ of a graph $G$ on $n$ vertices, then $pc(G)=2$ ?
\end{problem}

\begin{problem}\label{problem2}
 Is there a constant $\frac{1}{4} < \beta < 1$, such that if $d(x)+d(y)\geq \beta n$ for every pair of nonadjacent vertices $x,y$ of a bipartite graph $G$ on $n$ vertices, then $pc(G)=2$ ?
\end{problem}

This kind of conditions is usually called the {\it degree sum condition}. Our main result in this paper is devoted to studying degree sum condition for a general graph or a bipartite graph to have proper connection number 2. As a result, the following conclusions are obtained.

\begin{thm}\label{thm3.1}
Let $G$ be a connected noncomplete graph of order $n\geq 5$ with $G \notin \{G_1,G_2,G_3\}$ shown in Figure 1. If $d(x)+d(y)\geq \frac{n}{2}$ for eevery pair of nonadjacent vertices $x,y\in V(G)$, then $pc(G)=2$.
\end{thm}

\begin{figure}[h!]
\centering
\includegraphics[width=1.0\textwidth]{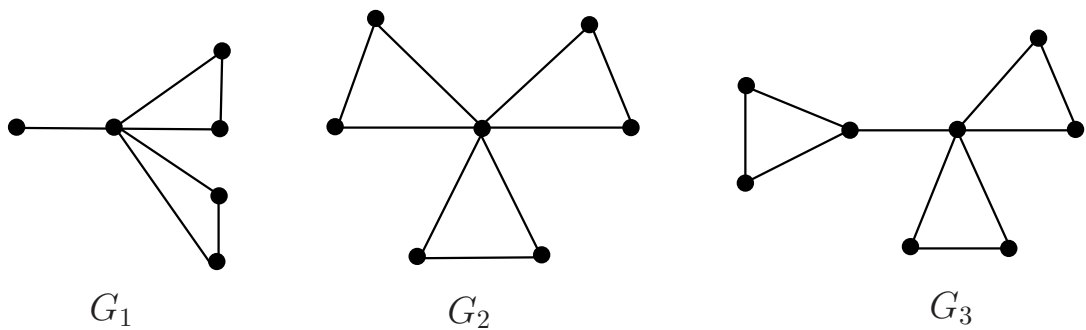}
\caption{Three graphs of Theorem \ref{thm3.1} }
\end{figure}

\begin{thm}\label{thm4.1}
Let $G$ be a connected bipartite graph of order $n\geq4$. If $d(x)+d(y)\geq \frac{n+6}{4}$ for every pair of nonadjacent vertices $x,y\in V(G)$, then $pc(G)=2$.
\end{thm}

The following example shows that the degree sum condition in Theorems \ref{thm3.1} are best possible. Let $G_1$ be a complete graph on $3$ vertices and $G_i$ be a complete graph with $\frac{n-3}{2}$ vertices ($n\geq5$) for $i=2,3$. Then, take a vertex $v_i\in G_i$ for each $1\leq i \leq 3$. Let $G$ be a graph obtained from $G_1\cup G_2 \cup G_3$ by joining $v_1$ and $v_i$ with an edge for $i=2,3$. It is easy to see that $d(x)+d(y)\geq \frac{n}{2}-1$ for every pair of nonadjacent vertices $x,y\in V(G)$, but $pc(G)=3$.

To show that the degree condition in Theorem \ref{thm4.1} is best possible, we construct the following example. Let $G_i$ be a complete bipartite graph such that each part has $3$ vertices and $v_0$ a vertex not in $G_i$ for $i=1,2,3$. Then, take a vertex $v_i\in G_i$ for each $1\leq i \leq 3$. Let $G$ be a bipartite graph obtained from $G_1\cup G_2 \cup G_3$ by joining $v_0$ and $v_i$ with an edge for each $1\leq i \leq 3$. It is easily checked that $d(x)+d(y)\geq \frac{n}{4}+1$ for every pair of nonadjacent vertices $x,y\in V(G)$, but $pc(G)=3$.

\section{Preliminaries}

At the beginning of this section, we present some basic concepts as follows.

\begin{df}\label{df1}
Given a colored path $P=v_1v_2\ldots v_{t-1}v_t$ between any two vertices $v_1$ and $v_t$, we define start$(P)$ as the color of the first edge $v_1v_2$ in the path, and define end$(P)$ as the color of the last edge $v_{t-1}v_t$. In particular, if $P$ is just the edge $v_1v_t$, then $start(P)=end(P)=c(v_1v_t)$.
\end{df}

\begin{df}\label{df2}
Let $c$ be a proper connection coloring of $G$. We say that $G$ has the strong property under $c$ if for any pair of vertices $u,v\in V(G)$, there exist two proper paths $P_1,P_2$ from $u$ to $v$ (not necessarily disjoint) such that $start(P_1)\neq start(P_2)$ and $end(P_1)\neq end(P_2)$.
\end{df}

\begin{df}\label{df3}
Let $B\subseteq E$ be the set of cut-edges of a graph $G$. We denote by $\mathcal{D}$ the set of connected components of $G^{'}=(V,E\setminus B)$. There exist two types of elements in $\mathcal{D}$, singletons and connected bridgeless subgraphs of $G$. we construct a new graph $G^{*}$ that is obtained from contracting each element of $\mathcal{D}$ of $G^{'}$ to a vertex. It is well-known that $G^{*}$ is called the {\it bridge-block tree} of $G$. For the sake of simplicity, we call every element of $\mathcal{D}$ a {\it block} of $G$. In particular, an element of $\mathcal{D}$ which corresponds to a leaf in $G^{*}$ is called a {\it leaf-block} of $G$.
\end{df}

\begin{df}\label{df4}
A Hamiltonian path in a graph $G$ is a path containing every vertex of $G$. And a graph having a Hamiltonian path is called a traceable graph.
\end{df}

\begin{df}\label{df5}
Let $G$ be a graph with vertex set $V$. A vertex partition $V=V_1\cup V_2\cup\cdots \cup V_k$ is called equitable if any two parts differ in size by at most one.
\end{df}

Next, we state some fundamental results on proper connection coloring    which are used in the sequel.

\begin{lem}\label{lem2.1}\cite{ALLZ}
If $G$ is a nontrivial connected graph and $H$ is a connected spanning subgraph of $G$, then $pc(G)\le pc(H)$. In particular, $pc(G)\le pc(T)$ for every spanning tree $T$ of $G$.
\end{lem}

\begin{lem}\label{lem2.4}\cite{ALLZ}
If $G$ be a traceble graph that is not complete, then $pc(G)=2$.
\end{lem}

\begin{lem}\label{lem2.3}\cite{ALLZ}
Let $G$ be a connected graph and $v$ a vertex not in $G$. If $pc(G)=2$, then $pc(G\cup v)=2$ as long as $d(v)\geq 2$, that is, there are at least two edges connecting $v$ to $G$.
\end{lem}

\begin{lem}\label{lem2.2}\cite{BFGMMMT}
If $G$ is a bipartite connected bridgeless graph, then $pc(G)= 2$. Furthermore, there exists a 2-edge-coloring $c$ of $G$ such that $G$ has the strong property under $c$.
\end{lem}

The following result is an immediate consequence of Lemma \ref{lem2.2}.

\begin{cor}\label{cor2.1}
Let $G$ be a connected bipartite graph and $G^*$ be the bridge-block tree of $G$. If $\Delta(G^*)\leq 2$, then $pc(G)=2$.
\end{cor}
\begin{proof}
The proof proceeds by induction on the number of bridges of $G$. If $G$ has no bridge, the result trivially holds by Lemma \ref{lem2.2}. Suppose that the result holds for every connected bipartite graph $H$ with $r\geq0$ bridges and $\Delta(H^*)\leq 2$. Let $G$ be a connected bipartite graph with $r+1$ bridges and $\Delta(G^*)\leq 2$. Note that $G^*$ is a path. Thus, assume that $G^*=B_1b_1B_2b_2\ldots B_rb_rB_{r+1}b_{r+1}B_{r+2}$ is a path, where $b_i$ is a bridge of $G$ for $i=1,2,\cdots,r+1$ and $B_i$ is a block of $G$ for $i=1,2,\cdots,r+2$. Let $H$ be the subgraph of $G$ such that $H^*=B_1b_1B_2b_2\ldots B_rb_rB_{r+1}$. By the induction hypothesis, there exists a proper connection 2-coloring $c$ of $H$ with colors \{1,2\}. In order to form a proper connection 2-coloring of $G$ with colors \{1,2\}, we only need to color the edges in $E(G)\setminus E(H)$. Without loss of generality, assume that $c(b_r)=1$ under $c$. Let $b_r=uv$ and $b_{r+1}=wz$ with $v,w\in B_{r+1}$. If $v=w$, then color the edge $b_{r+1}$ with color 2. If $v\neq w$, then there exists a proper path $P$ between $v$ and $w$ in $B_{r+1}$ under $c$, such that $start(P)=2$. Next, we color the edge $b_{r+1}$ satisfying $c(b_{r+1})\neq end(P)$. At last, if $B_{r+2}$ is not a singleton, applying Lemma \ref{lem2.2} to $B_{r+2}$, there exists a proper connection 2-coloring $c'$ of $B_{r+2}$ with colors \{1,2\} such that $B_{r+2}$ has the strong property under $c'$.
\end{proof}

\begin{thm}\label{thm2.4}\cite{HLQ}
Let $G$ be a connected noncomplete graph of order $n\ge5$. If $G \notin \{G_2,G_3\}$ shown in Figure 1, and $\delta(G)\geq n/4$, then $pc(G)=2$.
\end{thm}

\begin{thm}\label{thm2.5}\cite{HLQ}
Let $G$ be a connected bipartite graph of order $n\ge4$. If $\delta(G)\geq\frac{n+6}{8}$, then $pc(G)=2$.
\end{thm}

\section{ Proof of Theorem \ref{thm3.1} }

\begin{proof}
The result trivially holds for $\delta(G)\geq\frac{n}{4}$ by Theorem \ref{thm2.4}. Next, we only need to consider $\delta(G)<\frac{n}{4}$ in the following. Let $X=\{x\mid d(x)=\delta(G)\}$. If $n=5$, then $\delta(G)=1$. Since $d(x)+d(y)\geq3$ for every pair of nonadjacent vertices $x,y\in V(G)$, it follows that $G$ has a Hamiltonian path. Thus, $pc(G)=2$ by Lemma \ref{lem2.4}. If $n=6$, then $\delta(G)=1$. Take a vertex $x_0$ with $d(x_0)=1$. Let $N(x_0)=\{y_0\}$, and $Y=V(G)\setminus \{x_0,y_0\}=\{y_1,y_2,y_3, y_4\}$. Since $d(x_0)+d(y_i)\geq3$ for $i=1,\cdots4$, we have that $d(y_i)\geq2$ for $i=1,\cdots4$. If there exists some $y_i$ with $d(y_i)\geq3$ for $1\leq i\leq4$, then it is easy to check that $pc(G)=2$. If $d(y_i)=2$ for $i=1,\cdots4$, then $G=G_1$ in Figure 1, a contradiction. Hence, it is sufficient to prove that the result holds for $\delta(G)<\frac{n}{4}$ and $n\geq7$. Note that if $G$ contains a bridgeless bipartite spanning subgraph $H_0$, then $pc(G)\leq pc(H_0)\leq2$ by Lemmas \ref{lem2.1} and \ref{lem2.2}. Hence, assume that every bipartite spanning subgraph of $G$ has a bridge. Let $H$ be a bipartite spanning subgraph of $G$ such that $H$ has the maximum number of edges, and $\Delta(H^{*})$ is as small as possible in the second place. If $\Delta(H^{*})\leq2$, $pc(G)\leq pc(H)=2$ by Corollary \ref{cor2.1}. Next, assume that $\Delta(H^{*})\geq3$. To prove our result, we present the following fact.

\noindent {\bf Fact 1}. Let $e=u_1u_2$ be a cut-edge of $H$, and let $I_1$ and $I_2$ be two components of $H-e$. Then $\left|E_G(I_1,I_2)\right|\leq2$.

Suppose this is not true. Let $(U_i,V_i)$ be the bipartition of $I_i$ for $i=1,2$, such that $u_1\in U_1$ and $u_2\in U_2$. Noticing that $n\geq7$, it is possible that there exists only one part $U_i$ with $U_i=\{u_i\}$ and the corresponding $V_i=\emptyset$. Assume that there exists an edge $e_1\in (E_G(U_1,U_2)\cup E_G(V_1,V_2))\setminus e$, or there exist two edges $e_2,e_3\in E_G(U_1,V_2)\cup E_G(V_1,U_2)$. Let $H_1=H+e_1$ or $H_2=H-e+e_2+e_3$. It follows that $H_i$ has $\left|E(H))\right|+1$ edges for $i=1,2$, which contradicts the choice of $H$. Hence, $\left|E_G(I_1,I_2)\right|\leq2$.

Let $L$ be a leaf-block of $H$ and $b_L$ be the unique bridge incident with $L$ in $H$. Applying Fact 1 to the cut-edge $b_L$, it follows that $\left|E_G(L,G\setminus V(L))\right|\leq2$. Since $b_L\in E_G(L,G\setminus V(L))$, it is obtained that $\left|E_G(L,L')\right|\leq1$ for each pair of leaf-blocks $L,L'$ of $H$. In order to complete our proof, we consider the following two cases.

\textbf{Case 1.} If $V(L)\cap X =\emptyset$ for any leaf-block $L$ of $H$, then there exists a non-leaf block $B_\delta$ of $H$ such that $x_0\in B_\delta\cap X$. In this case, we claim that every leaf-block of $H$ is not a singleton. Suppose it is not true. Let $L_0$ be a leaf-block of $H$ with $V(L_0)=\{v_0\}$. It follows from Fact 1 that $d(v_0)\leq2$, and $v_0\notin X$. On the other hand, it is known that $d(v)\geq2$ for each vertex $v$ of each non-leaf block of $H$, which is impossible. Since every leaf-block of $H$ is a maximal connected bridgeless bipartite subgraph, every leaf-block of $H$ has at least 4 vertices. Let $L$ be a leaf-block of $H$. Note that $\left|E_G(L,G\setminus V(L))\right|\leq2$.  Then, take a vertex $v_L$ of $L$ that is not adjacent to $x_0$ and $N(v_L)\subseteq V(L)$. Thus, $d(v_L)\geq\frac{n}{2}-d(x_0)=\frac{n}{2}-\delta(G)$, which implies that $\left|L\right|\geq\frac{n}{2}-\delta(G)+1$. It follows that there exist at most three leaf-blocks of $H$. Otherwise, $\left|G\right|\geq 4 \times(\frac{n}{2}-\delta(G)+1)+1>n+5$, a contradiction. Hence, $\Delta(H^{*})=3$, and there is only one vertex $z_0\in V(H^{*})$ with $d_{H^{*}}(z_0)=3$. We define $B_0$ as the block of $H$ corresponding to $z_0$.

\textbf{Subcase 1.1.} If $B_0$ is a singleton, we let $V(B_0)=\{b_0\}$ and $e$ be a bridge incident to $b_0$. Suppose that $I_1$ and $I_2$ are two components of $H-e$. Without loss of generality, assume that $b_0\in I_2$, and let $L$ be the leaf-block in $I_1$. Bear in mind that $\left|E_G(I_1,I_2)\right|\leq2$. If $d_{I_1}(b_0)=1$, in this case we call $e$, the bridge incident to $b_0$, a bridge of type $\uppercase\expandafter{\romannumeral1}$. Then $\left|I_1\right|\geq\left|L\right|\geq\frac{n}{2}-\delta(G)+1$. If $d_{I_1}(b_0))=2$, in this case we call $e$, the bridge incident to $b_0$, a bridge of type $\uppercase\expandafter{\romannumeral2}$. Suppose that $I_1$ contains at least two blocks. Then $\left|I_1\right|\geq\left|L\right|+1\geq\frac{n}{2}-\delta(G)+2$. Suppose that $I_1$ contains only $L$. Then there exist two vertices $v_1,v_2$ of $L$ such that $b_0$ is adjacent to both $v_1$ and $v_2$.

\noindent {\bf Claim 1}. $\left|L\right|\geq\frac{n}{2}-\delta(G)+2$ for $L$ defined as above.

Suppose it is not true. Assume that $\left|L\right|\leq\frac{n}{2}-\delta(G)+1$. Let $V(L)=\{v_1,v_2, u_1,\ldots,u_t\}$ with $2\leq t\leq\frac{n}{2}-\delta(G)-1$. Since $\left|E_G(L,G\setminus V(L))\right|\leq2$, and $v_ib_0\in E_G(L,G\setminus V(L))$ for $i=1,2$, it follows that $u_i$ is not adjacent to $x_0$ and $N(u_i)\subseteq V(L)$ for $i=1,2$, which implies that $d_{L}(u_i)\geq\frac{n}{2}-d(x_0)=\frac{n}{2}-\delta(G)$. Thus, $\left|L\right|=\frac{n}{2}-\delta(G)+1$, and $u_i$ is adjacent to all other vertices of $L$ for $i=1,\cdots,t$. We can construct a new bipartite spanning subgraph $H'$ of $G$ by adding $b_0$ into $L$, such that $b_0$ and $v_i$ lie in distinct equitable parts and are adjacent in the new block for $i=1,2$, which contradicts the maximality of $H$.

It follows that $\left|I_1\right|\geq\left|L\right|\geq\frac{n}{2}-\delta(G)+2$. Let $k$ be the number of bridges incident to $b_0$ of type $\uppercase\expandafter{\romannumeral2}$. Then $\delta(G)\leq d(b_0)\leq k+3$. As a result, $\left|G\right|\geq 1+3\times(\frac{n}{2}-\delta(G)+1)+k>n+1$, a contradiction.

\textbf{Subcase 1.2.} If $B_0$ is not a singleton, since $B_0$ is a maximal connected bridgeless bipartite subgraph, $B_0$ has at least 4 vertices. Noticing that $d_{H^*}(z_0)=3$, it is obtained that $\left|E_G(B_0,G\setminus V(B_0))\right|\leq6$ by Fact 1. Then there exists at least one vertex $b$ in $B_0$ satisfying that all but at most one of the neighbours of $b$ are contained in $B_0$. Hence, $\left|B_0\right|\geq d(b)\geq\delta(G)$. Consequently, $\left|G\right|\geq \left|B_0\right|+3\times(\frac{n}{2}-\delta(G)+1)\geq\frac{3}{2}n-2\delta(G)+3>n+3$, a contradiction.

\textbf{Case 2.} There exists a leaf-block $L_0$ of $H$ such that $x_0\in L_0\cap X$.

\noindent {\bf Claim 2}. Let $L'$ be a leaf-block of $H$ with
$E_G(L_0,L')=\emptyset$. Then, $L'$ is not a singleton.

Suppose this is not true. Let $L'$ be a leaf-block of $H$, such that $E_G(L_0,L')=\emptyset$ and $V(L')=\{v'\}$. It follows from Fact 1 that $d(v')\leq2$. On the other hand, since $x_0$ is not adjacent to $v'$, we have that $d(v')\geq\frac{n}{2}-d(x_0)=\frac{n}{2}-\delta(G)$. Note that $n\geq7$ and $\delta(G)<\frac{n}{4}$. If $n=7$, then $\delta(G)=1$. Hence, $d(v')\geq\frac{n}{2}-1=\frac{5}{2}$, a contradiction. If $n\geq8$, then $d(v')\geq\frac{n}{2}-d(x_0)=\frac{n}{2}-\delta(G)>\frac{n}{4}\geq2$, that is $d(v')\geq3$, a contradiction.

Note that there exist at most two other leaf-blocks $L_1,L_2$ of $H$. Otherwise, there exist three other leaf-blocks $L_1,L_2,L_3$ of $H$. Since $\left|E_G(L_0,G\setminus V(L_0))\right|\leq2$, there exist at least two leaf-blocks of $L_1,L_2,L_3$, say $L_1,L_2$, such that $E_G(L_0,L_i)=\emptyset$ for $i=1,2$. By Claim 2, $L_i$ is not a singleton for $i=1,2$. Take a vertex $v_i$ of $L_i$ such that $v_i$ is not adjacent to $x_0$ and $N(v_i)\subseteq V(L_i)$ for $i=1,2$. Thus, $\left|L_i\right|\geq\frac{n}{2}-\delta(G)+1$ for $i=1,2$. It is easy to see that any leaf-block other than $L_i,L_j$ has at least $\delta(G)-1$ vertices. As a result, $\left|G\right|\geq 2\times(\frac{n}{2}-\delta(G)+1)+2(\delta(G)-1)+1=n+1$, a contradiction. Hence, $\Delta(H^{*})=3$, and there is only one vertex $z_0\in V(H^{*})$ with $d_{H^*}(z_0)=3$. We define $B_0$ as the block of $H$ corresponding to $z_0$. Let $C_0,C_1,C_2$ be the connected components of $H-V(B_0)$ such that $L_i$ is the leaf-block contained in $C_i$ for $i=0,1,2$. Suppose that there exist two leaf-blocks $L_i,L_j$ of $H$ such that $E_G(L_i,L_j)=\{f\}$ for $0\leq i\neq j\leq2$. Let $e_i$ be the unique bridge incident with both $B_0$ and $C_i$ in $H$. Let $H_1=H-e_i+f$. Note that $H_1$ is also a maximum bipartite spanning subgraph of $G$, but $\Delta (H_1^*)=2$, which contradicts the choice of $H$. Thus, it is obtained that $ E_G(L_i,L_j)=\emptyset$ for any two leaf-block $L_i,L_j$ for $0\leq i\neq j\leq2$.

\textbf{Subcase 2.1.} If $\delta(G)\leq2$, then $d(x_0)=\delta(G)\leq2$. Since $ E_G(L_0,L_i)=\emptyset$ for $1\leq i\leq2$, with the help of Claim 2, $L_i$ is not a singleton for $i=1,2$. Consequently, $\left|L_i\right|\geq\frac{n}{2}-\delta(G)+1$ for $i=1,2$. It follows that $\left|G\right|\geq \left|\{x_0\cup N(x_0)\}\right|+\left|L_1\right|+\left|L_2\right|=(1+\delta(G))+2\times(\frac{n}{2}-\delta(G)+1)=n+3-\delta(G)\geq n+1$, a contradiction.

\textbf{Subcase 2.2.} If $\delta(G)\geq3$, then $d(x_0)\geq3$. First, we obtain that $L_0$ is not a singleton, which implies that $\left|L_0\right|\geq \delta(G)+1$. Next, since $ E_G(L_0,L_i)=\emptyset$ for $1\leq i\leq2$, it follows from Claim 2 that $L_i$ is not a singleton for $i=1,2$, and so $\left|L_i\right|\geq\frac{n}{2}-\delta(G)+1$ for $i=1,2$. Consequently, every leaf-block of $H$ is not a singleton. With the similar argument in Case 1, we distinguish two cases based on the condition that $B_0$ is or not a singleton. The unique different point is that there exists one leaf-block $L_0$ of $H$ with $\left|L_0\right|\geq \delta(G)+1$ in this case, and $\left|L\right|\geq\frac{n}{2}-\delta(G)+1$ for each leaf-block $L$ of $H$ in Case 1. But the unique different point has no influence on proving our result. If $B_0$ is a singleton, it is worth mentioning that if the leaf-block $L$ in Claim 1 is exactly $L_0$, then the corresponding result is changed to $\left|L_0\right|\geq \delta(G)+2$ in parallel. Using the similar argument in Subcase 1.1, we can obtain that $\left|G\right|\geq 1+2\times(\frac{n}{2}-\delta(G)+1)+(\delta(G)+1)+k\geq n+1$, where $k\geq\delta(G)-3$, a contradiction. If $B_0$ is not a singleton, using the similar argument in Subcase 1.2, we can deduce that $\left|G\right|\geq \left|B_0\right|+2\times(\frac{n}{2}-\delta(G)+1)+(\delta(G)+1)\geq n+3$, where $\left|B_0\right|\geq\delta(G)$, a contradiction.

\end{proof}

\section{Proof of Theorem \ref{thm4.1}}

\begin{proof}
The result trivially holds for $\delta(G)\geq\frac{n+6}{8}$ by Theorem \ref{thm2.5}. Next, we only need to consider $n\geq4$ and $\delta(G)<\frac{n+6}{8}$ in the following. Let $G^{*}$ be the bridge-block tree of $G$. If $\Delta(G^{*})\leq2$, then $pc(G)\leq 2$ by Corollary \ref{cor2.1}. Next, assume that $\Delta(G^{*})\geq3$ and let $X=\{x\mid d(x)=\delta(G)\}$.

\textbf{Case 1.} If $V(L)\cap X =\emptyset$ for any leaf-block $L$ of $G$, then there exists a non-leaf block $B_\delta$ of $G$ such that $x_0\in B_\delta\cap X$. In this case, we claim that every leaf-block of $G$ is not a singleton. Suppose it is not the case. Let $L_0$ be a leaf-block of $G$ with $V(L_0)=\{v_0\}$. It follows that $d(v_0)=1$, but $v_0\notin X$, which is a contradiction. Since every leaf-block of $G$ is a maximal connected bridgeless bipartite subgraph, every leaf-block of $G$ has at least 4 vertices. Let $L$ be a leaf-block of $G$. Note that $\left|E_G(L,G\setminus V(L))\right|=1$. Then, take a vertex $v_L$ of $L$ that is not adjacent to $x_0$ and $N(v_L)\subseteq V(L)$ (there are at least $\left|L\right|-1$ such vertices). Thus, $d(v_L)\geq\frac{n+6}{4}-d(x_0)=\frac{n+6}{4}- \delta(G)$. Since each part of $L$ contains at least two vertices, which implies that  $\left|L\right|\geq\frac{n+6}{2}-2 \delta(G)$. It follows that there exist at most three leaf-blocks of $G$. Otherwise, $\left|G\right|\geq 4 \times(\frac{n+6}{2}-2\delta(G))+1>n+7$, a contradiction. Hence, $\Delta(G^{*})=3$, and there is only one vertex $z_0\in V(G^{*})$ with $d_{G^{*}}(z_0)=3$. We define $B_0$ as the block of $G$ corresponding to $z_0$. If $B_0$ is a singleton, then $\delta(G)\leq d(b_0)= 3$. If $B_0$ is not a singleton, since $B_0$ is a maximal connected bridgeless bipartite subgraph, $B_0$ has at least 4 vertices. Consider the subgraph $B_0$. We have that $\delta(B_0)\geq \delta(G) -3$. Noticing that $d_{G^*}(z_0)=3$, then there exists at least one vertex $b$ in $B_0$ satisfying that all the neighbors of $b$ are contained in $B_0$. Hence, $\left|B_0\right|\geq\delta(G)+(\delta(G)-3)=2\delta(G)-3$. Thus, no matter whether $B_0$ is or not a singleton, it always holds that $\left|B_0\right|\geq 2 \delta(G) -5$. As a result, $\left|G\right|\geq 3 \times(\frac{n+6}{2}-2\delta(G))+\left|B_0\right| > n+1$, a contradiction.

\textbf{Case 2.} There exists a leaf-block $L_0$ of $G$ such that $x_0\in L_0\cap X$. If $L_0$ is a singleton, then $\delta(G)=d(x_0)=1$. Let $L$ be another leaf-block of $G$. Then $L$ can not be a singleton. Suppose to the contrary, assume that $V(L)=\{v_0\}$, that is $d(v_0)=1$. Noticing  that $x_0$ is not adjacent to $v_0$, then $d(x_0)+d(v_0)\geq\frac{n+6}{4}\geq\frac{5}{2}$, which is impossible. Hence, $L$ has at least 4 vertices. It follows from $\Delta(G^{*})\geq3$ from that there exist at least two other leaf-blocks $L_1,L_2$ of $G$. Since $E_G(L_0, L_i)=\emptyset$, $d(v_i)\geq\frac{n+6}{4}-d(x_0)=\frac{n+6}{4}-1$ for each vertex $v_i$ of $L_i$ for $i=1,2$, this means that $\left|L_i\right|\geq\frac{n+6}{2}-2$. Thus, $\left|G\right|\geq \left| \{x_0\cup N(x_0)\}\right|+\left|L_1\right|+\left|L_2\right|\geq 1+1+2 \times(\frac{n+6}{2}-2)=n+4$, a contradiction. If $L_0$ is not a singleton, since $L_0$ is a maximal connected bridgeless bipartite subgraph, $L_0$ has at least 4 vertices. In this case, it is easy to check that $\left|L_0\right|\geq 2\delta(G)$. Let $L$ be another leaf-block of $G$. We claim that $L$ is not a singleton. Suppose to the contrary, assume that $V(L)=\{v_0\}$, which implies $d(v_0)=1$. Since $E_G(L_0, L)=\emptyset$, $d(v_0)\geq\frac{n+6}{4}-d(x_0)=\frac{n+6}{4}-\delta(G)>\frac{n+6}{8}\geq\frac{5}{4}$, a contradiction. Hence $L$ has at least 4 vertices. Furthermore, $d(v)\geq\frac{n+6}{4}-d(x_0)=\frac{n+6}{4}-\delta(G)$ for each vertex $v$ of $L$, which implies that $\left|L\right|\geq\frac{n+6}{2}-2\delta(G)$. It follows that there exist at most two leaf-blocks of $G$ other than $L_0$. Otherwise, $\left|G\right|\geq 1+2\delta(G)+3 \times(\frac{n+6}{2}-2\delta(G))>n+7$, a contradiction. Hence, $\Delta(G^{*})=3$, and there is only one vertex $z_0\in V(G^{*})$ with $d_{G^{*}}(z_0)=3$. We define $B_0$ as the block of $G$ corresponding to $z_0$. With similar argument in Case 1, we consider two cases based on the condition whether $B_0$ is or not a singleton. One can find that no matter what case occurs, it always holds that $\left|B_0\right|\geq 2 \delta(G) -5$. As a result, $\left|G\right|\geq 2 \times(\frac{n+6}{2}-2\delta(G))+2\delta(G)+\left|B_0\right| \geq n+1$, which is impossible.
\end{proof}

\end{document}